\documentclass[12pt]{article}
\usepackage{amsmath,amsthm,amsfonts,amssymb,amscd,amstext,mathrsfs}
\theoremstyle{plain}

\usepackage[english]{babel}

\newcommand{\bC}{\mathbb C}
\newcommand{\bN}{\mathbb N}

\newcommand{\bP}{\mathbb P}

\newcommand{\cM}{\mathcal M}

\newcommand{\cO}{\mathcal O}

\newcommand{\ra}{\rightarrow}

\newtheorem{definition}{Definition}
\newtheorem{thm}[definition]{Theorem}

\newtheorem{prop}[definition]{Proposition}
\newtheorem{conj}[definition]{Conjecture}
\newtheorem{claim}[definition]{Claim}

\numberwithin{equation}{section}
\numberwithin{definition}{section}

\begin{document}
\baselineskip=17pt
\title{On   a theorem of Wiegerinck}

\author{R\'obert Sz\H{o}ke\\
Department of Analysis, Institute of Mathematics\\
ELTE E\"otv\"os Lor\'and University\\
P\'azm\'any P\'eter s\'et\'any 1/c, Budapest 1117, Hungary\\
E-mail: rszoke@caesar.elte.hu\\
ORCiD:0000-0002-8723-1068}

\date{}
\maketitle




\begin{abstract}
  A theorem of Wiegerinck says that the Bergman space over
  any domain  in $\bC$ is either trivial or infinite dimensional. We generalize
  this theorem in the following form. Let E be a Hermitian,
  holomorphic vector bundle over
  $\bP^1$, the later equipped with a volume form and
  $D$ an arbitrary domain in $\bP^1$. Then the space
  of holomorphic $L^2$ sections of $E$ over $D$ is either equal to
  $H^0(M,E)$ or
 it has infinite dimension.

\end{abstract}

\section{Introduction}\label{S:intr}

Let $D$ be a domain in $\bC^n$ and $\cO L^2(D)$ the Bergman space, that is the
space of those holomorphic functions on $D$ that also belong to $L^2(D)$,
with respect to the Lebesgue measure.
Wiegerinck proved in \cite{W84}, that for $n=1$, and $D$ any domain
in $\bC$, either
$\cO L^2(D)$ is trivial ($=\{0\}$) or  has infinite dimension.
In the same paper for any $k>0$ he constructed a Reinhardt domain in $\bC^2$
with $k$-dimensional Bergman space. These examples were not logarithmically
convex, hence they were not Stein domains.
These results give rise to the following natural conjecture.

\begin{conj}(Wiegerinck)\label{C:1} Let $D$ be a Stein domain in $\bC^n$, then its Bergman space is either
  trivial or infinite dimensional.
\end{conj}
Despite  of a lot of work and partial results (cf.
\cite{BZ20, D07, GHH17, J12, PW07, PW17}, as far as we can tell, the
conjecture is still open.

The purpose of this paper is to show that there is another very natural
direction in which Wiegerinck's result could be generalized.
More precisely
we mean to consider the following situation.
Let  $M$ be a compact, connected Riemann surface, $E\ra M$ a holomorphic vector
bundle equipped with  a smooth Hermitian metric $h$, and
$dV$ a smooth volume form on $M$. For  a domain  $D\subset M$, denote by
$\cO L^2(\left.E\right|_D)$ 
the set of those holomorphic  sections of
$E$ over $D$ for which
$$
\|s\|^2:=\int\limits_Dh(s,s)dV<\infty.
$$

As is well known $\cO L^2(\left.E\right|_D)$ is a Hilbert space
and  the vector space $H^0(M,E)$
of global holomorphic sections of $E$ is finite dimensional (\cite{GH}).
Clearly
$H^0(M,E)\subset \cO L^2(\left.E\right|_D)$.
\begin{conj}\label{C:b}
 $\cO L^2(\left.E\right|_D)$ is either equal to  $H^0(M,E)$ or
 it has infinite dimension.
 \end{conj}

Supporting this conjecture we have the following observations.
First of all Wiegerinck's result in \cite{W84} implies  that
Conjecture \ref{C:b} is true when
$M=\bP^1$ and $E$ is the canonical bundle of $\bP^1$.
It is  also true when $M$ and $E$ are arbitrary and $K$ is either small,
i.e.,  locally polar
(cf. Proposition~\ref{P:finite}) or
$K$ is big, i.e., has nonempty interior (cf. Proposition~\ref{P:infty}).

Our main result is that for $M=\bP^1$, Conjecture \ref{C:b}
holds. 

\begin{thm}\label{T:main}

  Let   $M=\bP^1$ and $D$ and $E$ be arbitrary. Then
  Conjecture \ref{C:b} is true.
\end{thm}

\section{Proofs}

Suppose $M$ is an arbitrary  Riemann surface. Recall that  a subset
$K\subset M$ is called locally polar, if for each point $p\in K$, there exists
a connected neighborhood $U$ of $p$ and a subharmonic function
$\varphi\not\equiv-\infty$ on $U$ such that
$K\cap U\subset\{y\in U | \varphi(y)=-\infty\}$.
Now back to our original setting, let us assume that $M$ is a compact
Riemann surface, $(E,h)\rightarrow M$ a holomorphic Hermitian vector bundle
over $M$, $dV$ a smooth volume form on $M$, $D\subset M$ a domain and
$K=M\setminus D$. When $K$ is small, every element of
$\cO L^2(\left.E\right|_D)$ extends holomorphically to a global section of $E$
yielding:

\begin{prop}\label{P:finite} Suppose  $K$ is locally polar. Then
  $\cO L^2(\left.E\right|_D)=H^0(M,E)$.
\end{prop}
The special case when $M$ is $\bP^1$ and $E$ is its canonical bundle, this
is a reformulation of a result of Carleson in \cite{Ca67}. Indeed Carleson
proves that in this case
the Bergman space over $D=\bP^1\setminus K$ is trivial. On the
other hand the canonical bundle of $\bP^1$ is negative so its only
holomorphic section is the zero section, i.e., $H^0(M,E)$ is  trivial
as well.
The general case is also essentially
known, but for the readers' sake we include a proof here.

\begin{proof}
  \footnote{We thank László Fehér for discussions on totally disconnected
  spaces.}
  We need to show that each element of
  $\cO L^2(\left.E\right|_D)$ extends holomorphically
  to yield a global holomorphic section of $E$. 
  It suffices to prove  local extension.
  But locally the bundle $E$ is trivial and the volume form of $M$
  is equivalent to the Euclidean volume form (in a coordinate neighborhood).
  So the problem reduces to an extension problem on holomorphic functions.

  Let $W\subset \mathbb C$ be open, $X\subset W$  a 
   closed, locally polar set in $W$ and $f$ a holomorphic $L^2$ function on
   $W\setminus X$. We need to show that for each point $p\in X$, $f$
   extends holomorphically to a neighborhood of $p$.
   According to \cite[Ch. 21, Proposition 5.5]{Co95}, $X$ is  a
 polar  subset of $\mathbb C$.
     In light of \cite[Ch. 21,Theorem 9.5 (c)]{Co95}, it suffices
    to find an open set $V$, that contains $p$, so that $K=V\cap X$ is compact.
    To find such a set,
    one can argue as follows.
    Since $X$ is  a
 polar subset of $\mathbb C$, by \cite[Corollary 3.8.5]{R95}, $X$
 is  a totally disconnected space. Now $X$ is closed in $W$, hence it is also
 locally compact. Let $F\subset X$ be a compact neighborhood of $p$ and
 $U\subset X$ open with $p\in U\subset F$.
  A general result on totally disconnected, locally compact   spaces
  (\cite[Theorem 6.2.9]{E89}) tells us that there exists  a set
  $K\subset U$ that is both open and closed in $X$ and $p\in K$.
  Let  $V\subset\mathbb C$ be open 
  with $V\cap X=K$ (can assume $V\subset W$). 
  Since $K$ is  a closed subset of a compact space (namely $F$),
  it is also compact. 
 \end{proof}   

The other extreme case is when $K$ is big in the sense, that its interior is
nonempty.
\begin{prop}\label{P:infty} Suppose   the interior of $K$ is nonempty.
  Then $\cO L^2(\left.E\right|_D)$ is  infinite dimensional.
\end{prop}

\begin{proof}

Let $p$ be an interior point of $K$ and $l$ be the rank of $E$.
  Let
 $G=M\setminus\{p\} $.  
 Since $G$ is an open Riemann surface, $\left.E\right|_G$ is holomorphically
 trivial (\cite[Theorem 30.4]{F81}),
 hence $H^0(G,\left.E\right|_G)=Hol(G,\mathbb C^l)$, the
 space of holomorphic maps from $G$ to $\mathbb C^l$.
  $G$ being Stein, the latter space  is
 infinite dimensional.
 But $\overline D$ is a  compact subset in $G$, hence
 all these holomorphic maps (sections), via restriction  to $D$, belong to
 $L^2(\left.E\right|_D)$.
\end{proof}

 Denote by $H^\infty(D)$  the set of bounded holomorphic functions on $D$.
We shall say that $H^\infty(D)$ is nontrivial, if it contains an element not
identically constant.

\begin{prop}\label{P:b} Suppose $H^\infty(D)$ is nontrivial. Then 
  Conjecture \ref{C:b}
    holds.
\end{prop}

\begin{proof}

   Suppose $\exists s\in \cO L^2(\left.E\right|_D)$, with
  $s\not\in H^0(M,E)$. Let $g\in H^\infty(D)$ be a nonconstant element 
  and $w\in D$ be arbitrary. Then  $f:=g-g(w)\in H^\infty(D)$ vanishes at $w$
  with finite, nonzero multiplicity.
  This yields
$$
f^ms\in \cO L^2(\left.L\right|_D), \quad m\in\bN,
$$
implying the infinite dimensionality of  $\cO L^2(\left.L\right|_D)$.

  \end{proof}

For an arbitrary holomorphic line bundle $L\ra M$, denote by $\cM(L)$ the
set of meromorphic sections. Whereas for an open subset $D\subset M$
the set of meromorphic functions on $D$ is denoted by
$\cM(D)$.

\begin{proof}[Proof of Theorem~\ref{T:main}]

 Let $E\ra \bP^1$ be a holomorphic vector
  bundle of rank $r$.
  According to Grothendieck's theorem (\cite{G57}) $E$ holomorphically
  splits  as a direct sum of holomorphic line bundles:
  $$
E=L_1\oplus\dots \oplus L_r.
$$
Therefore it is enough to prove  Conjecture \ref{C:b} when $E=L$ is a
holomorphic
line bundle.
 The proof then follows similar reasonings as \cite{W84}.
  
 As is well known (\cite[Prop. 4.6.2]{V11}),
 there exists an integer $k$, so that
  $L=\cO(k)$, where $\cO(k)$ denotes the line bundle associated to the divisor
  $kp$, $p$ being an (arbitrary) point in $\bP^1$.
  Suppose $\exists s\in \cO L^2(\left.L\right|_D)$, with
  $s\not\in H^0(\bP^1,L)$.

  {\it  Case I.}   $s\in\cM(L)$.

Since isolated $L^2$ singularities are removable and
  $s\in\cM(L)\setminus H^0(\bP^1,L)$, we get

\begin{equation}\label{E:inf}
\int\limits_{\bP^1}h(s,s)dV=\infty.
\end{equation}
Now $\left.s\right|_D$ is in $L^2$, hence 

\begin{equation}\label{E:pos}
\lambda(\bP^1\setminus D)>0,
\end{equation}
where $\lambda$ denotes  Lebesgue measure on $\bP^1$. (Having Lebesgue measure
zero or positive is well defined on any smooth manifold and does not depend
on any choice of a Riemannian metric or a volume form.)
 From \eqref{E:pos} and a well known result
about the Cauchy transform (cf. \cite[p.2]{G72}) we get that
$H^\infty(D)$ is nontrivial, so in light of Proposition \ref{P:b}  we are done.

{\it Case II.} $s\not\in\cM(L)$ and $k\ge 0$.

 Pick an arbitrary point
$\infty\not=z_0\in D$.
 Since $L=\cO(k)$, $L$ admits  a holomorphic section $s_L:\bP^1\ra L$  with
the only possible zero at $z_0$, the multiplicity at $z_0$ being $k$. Then
$$
s=fs_L,
$$
with  some $f\in\cM(D)$. Here $f$ must be holomorphic in $D$ except perhaps at
$z_0$, where $f$ may have  a pole  at most of order $k$.
Since $s$ is not in  $\cM(L)$, $f$ cannot be in $\cM(\bP^1)$.
Let $z_1\in D\setminus\{z_0,\infty\}$ be another arbitrary point. Define the function $g_1$ by
\begin{equation}\label{E:new}
g_1(z)=(f(z)-f(z_1))\frac{(z-z_0)}{(z-z_1)}.
\end{equation}

\begin{claim}\label{CL:1} $s_1:=g_1s_L\in \cO L^2(\left.L\right|_D)$
\end{claim}
\begin{proof}

$g_1$ is holomorphic in $D\setminus \{z_0\}$ and  in $z_0$ it has
either a removable singularity or a pole of order at worst $k-1$. Hence
$g_1s_L$ is a holomorphic section of $L$ over $D$. Let $U\subset D$ be an open
subset with $z_0, z_1\in U$ and $\bar U\subset D$ compact. Holomorphicity of $g_1s_L$ implies that $g_1s_L$
is an $L^2$ section of $L$ over $U$.

Since $(f-f(z_1))s_L$ is an $L^2$ section over $D$ and
$$
\frac{z-z_0}{z-z_1}
$$
is bounded on $\bP^1\setminus U$, we get that $g_1s_L$ is in $L^2$ over
$D\setminus U$ as well. Since $g_1s_L$ is  holomorphic,
it is also in $L^2$ over the compact set $\bar U$ which proves
Claim~\ref{CL:1}.
\end{proof}

Back to the proof of our theorem,
let now $N$ be an arbitrary  positive integer and $z_1, z_2,\dots, z_N\in D\setminus\{z_0,\infty\}$ be arbitrary different
points. Denote the corresponding functions defined by formula \eqref{E:new}
by $g_j$ and let
 $s_j:=g_js_L,$ where $j=1,\dots, N$. Due
to  Claim \ref{CL:1} they all belong to $\cO L^2(\left.L\right|_D)$.

\begin{claim}\label{CL:2} $s_1,\dots, s_N$ are linearly independent.
\end{claim}

\begin{proof}
Indeed if they were linearly dependent, that would yield that $f$ is a
meromorphic function on $\bP^1$, a contradiction.
\end{proof}

Claim \ref{CL:2} shows that
$\cO L^2(\left.L\right|_D)$ has infinite dimension when $k\ge 0$,
finishing the proof of Case II.

{\it  Case III.}  $L=\cO(k)$, $k<0$,   $s\in \cO L^2(D,L)$,  and
$s\not\in\cM(L)$    

Similarly to  Case II, pick a point $\infty\not=z_0\in D$.
$L=\cO(k)$ implies that $L$ admits  a meromorphic section $m_L$
with $z_0$ being the only
singularity,  a pole of order $k$ and that $m_L$ does not vanish on
$\bP^1\setminus \{z_0\}$.

Because $s\in \cO L^2(\left.L\right|_D)$, we can write 
$$
s=fm_L,
$$
where $f$  now is holomorphic in $D$ but is not in $\cM(\bP^1)$.
Since $s$ is holomorphic,
$f$ must vanish at $z_0$. Denote by $l$
its multiplicity (necessarily $l\ge k$).
 Then $s$ vanishes at $z_0$ with multiplicity $l-k$.


Let $N>l$ be an arbitrary positive integer and 
 $z_1, z_2,\dots, z_N\in D\setminus\{z_0,\infty\}$  arbitrary but different
 points. Define the corresponding functions  $g_j$ by  formula \ref{E:new}.
 Then all $g_j$ are holomorphic in $D$, but $g_jm_L$ has a pole at $z_0$ of
 order  $k-1$ if $f(z_j)\not=0$.
 Let $\lambda_j\in\bC$, $j=1,\dots,N$, $\lambda=(\lambda_1,\dots,\lambda_N)$
 and define $g_\lambda$ by
 \begin{equation}\label{E:lambda}
   g_\lambda(z):=\sum\limits^N_{j=1}\lambda_jg_j(z)=
   f(z)\left(\sum\limits^N_{j=1}\frac{\lambda_j}{z-z_j}\right)(z-z_0)-
  \left(\sum\limits^N_{j=1}\frac{\lambda_jf(z_j)}{z-z_j}\right)(z-z_0).
 \end{equation}

 Since $f\not\in\cM(\bP^1)$, $g_\lambda\not\in\cM(\bP^1)$ except
  if all the $\lambda_j^,s$ are zero. 
 In \eqref{E:lambda} the first term is holomorphic near $z_0$ and vanishes at $z_0$ with
 multiplicity at least $l+1$. The function
 \begin{equation}\label{E:rec}
   f_\lambda:=\sum\limits^N_{j=1}\frac{\lambda_jf(z_j)}{z-z_j}
 \end{equation} is also holomorphic near $z_0$ and the coefficients of its
 Taylor series around
 $z_0$ depend linearly on the $\lambda_j^,s$. Therefore
 to impose the condition on
 $f_\lambda$
 to vanish at $z_0$ with multiplicity $l$ means solving
 a homogeneous system of linear equations in $N$ unkowns and $l$ equations.
 Since $N>l$, this system will have a nontrivial solution. 
 Let $\lambda$ be such a solution. Then the corresponding $g_\lambda$   will
 vanish at $z_0$ with multiplicity at least $l+1$. Hence
the holomorphic section  $s_\lambda:=g_\lambda m_L$
 vanishes at $z_0$ with higher multiplicity than $s$ does.
 A similar argument as in the proof of Claim \ref{CL:1} shows that
 $s_\lambda\in\cO L^2(\left.L\right|_D)$.
Moreover
  $g_\lambda\not\in\cM(\bP^1)$ implies that $s_\lambda\not\in\cM(L)$. Hence we can
   now repeat the whole process to $s_\lambda$. That yields indeed that
  $\cO L^2(\left.L\right|_D)$ has infinite dimension.
\end{proof}  

{\bf Acknowledgement:}
This research was   supported by NKFI grant K112703, K128862  and the R\'enyi Institute of
Mathematics. We thank  the referee for a very careful reading of the
manuscript and suggestions to improve the presentation of the paper.

{\bf Note added in proof.} During the refereeing process, a closely related
paper appeared in the arxiv: A-K. Gallagher, P. Gupta and L. Vivas:
On the dimension of Begman spaces on $\mathbb P^1$, arXiv:2110.02324v2. Here
the authors give potential theoretic characterizations of the dimension of
the Bergman space. Furthermore, by generalizing Wiegerinck's example,
they show that in higher dimensions the dichotomy
we considered in our paper, does
not hold in general for certain line bundles over $\mathbb P^2$.
\end{document}